\documentclass[11pt]{amsart}
\usepackage{amscd}
\usepackage[all]{xy}
\usepackage{graphicx}
\usepackage{amsmath}
\usepackage{amsfonts}
\usepackage{amssymb}
\usepackage{amsthm}
\usepackage{latexsym}
\usepackage{dsfont}
\usepackage{slashed}
\usepackage{enumerate}
\usepackage{bbm}
\newcommand{\kk}{{\mathbbm k}}

\usepackage{soul}
\usepackage{comment}
\usepackage{color}
\usepackage{rotating}
\usepackage{url}

\usepackage{mathrsfs}
\usepackage{verbatim}
\usepackage{mathtools}

\usepackage{longtable}
\usepackage[table]{xcolor}

\RequirePackage{etex}
\usepackage{dynkin-diagrams}
\tikzset{big arrow/.style={
    -Stealth,line cap=round,line width=1mm,
    shorten <=1mm,shorten >=1mm}}

\numberwithin{equation}{section}

\theoremstyle{plain}
\newtheorem{lemma}[equation]{Lemma}
\newtheorem{proposition}[equation]{Proposition}

\newtheorem{corollary}[equation]{Corollary}

\newtheorem*{proposition*}{Proposition}
\newtheorem*{lemma*}{Lemma}

\theoremstyle{definition}

\newtheorem{remark}[equation]{Remark}

\newcommand{\R}{{\mathds R}}
\newcommand{\C}{{\mathds C}}

\newcommand{\Z}{{\mathds Z}}

\newcommand{\ad}{{\rm ad \,}}
\newcommand{\Ad}{{\rm Ad }}

\newcommand{\z}{{\mathfrak z}}

\newcommand{\g}{{\mathfrak g}}
\newcommand{\gl}{{\mathfrak {gl}}}
\newcommand{\h}{{\mathfrak h}}

\newcommand{\ug}{{\mathfrak u}}  
\newcommand{\su}{\mathfrak{su}}

\newcommand{\ssl}{{\mathfrak {sl}}}

\renewcommand{\t}{{\mathfrak t }}

\newcommand{\lra}{\longrightarrow}

\newcommand{\GL}{{\rm GL}}
\newcommand{\SL}{{\rm SL}}

\newcommand{\hs}{\kern 0.8pt}
\newcommand{\hsss}{\kern 1.2pt}
\newcommand{\hm}{\kern -0.8pt}
\newcommand{\hmm}{\kern -1.2pt}

\newcommand{\hssh}{\kern 1.2pt}
\newcommand{\hshs}{\kern 1.6pt}
\newcommand{\hssss}{\kern 2.0pt}

\newcommand{\into}{\hookrightarrow}
\newcommand{\isoto}{\overset{\sim}{\to}}
\newcommand{\onto}{\twoheadrightarrow}
\newcommand{\labelto}[1]{\xrightarrow{\makebox[1.5em]{\scriptsize ${#1}$}}}

\newcommand{\upalpha}{\hs^\alpha\kern-0.5pt}

\newcommand{\real}{\bf}

\newcommand{\GG}{{\real G}}
\newcommand{\HH}{{\real H}}

\newcommand{\DD}{{\real D}}

\newcommand{\PP}{{\real P}}

\newcommand{\TT}{{\bf T}}
\newcommand{\UU}{{\real U}}
\newcommand{\VV}{{\real V}}

\newcommand{\YY}{{\real Y}}

\newcommand{\wt}{\widetilde}
\DeclareMathOperator{\Lie}{Lie}
\DeclareMathOperator{\Aut}{Aut}
\DeclareMathOperator{\Inn}{Inn}
\DeclareMathOperator{\Out}{Out}
\DeclareMathOperator{\Dyn}{{Dyn}}

\newcommand{\Hom}{{\rm Hom}}
\newcommand{\Gal}{{\rm Gal}}

\newcommand{\G}{{\mathbb G}}

\newcommand{\SmallMatrix}[1]{\text{{\tiny\arraycolsep=0.4\arraycolsep\ensuremath
    {\begin{pmatrix}#1\end{pmatrix}}}}}

\newcommand{\Ho}{{\mathrm{H}\kern 0.3pt}}
\newcommand{\Zl}{{\mathrm{Z}\kern 0.2pt}}
\newcommand{\Bd}{{\mathrm{B}\kern 0.2pt}}

\newcommand{\Stab}{{\rm Stab}}

\newcommand{\inn}{{\rm inn}}

\newcommand{\der}{{\rm der}}

\newcommand{\upgam}{{\hs^\gamma\hm}}

\newcommand{\EEE}{{\real E}}
\newcommand{\AAA}{{\real A}}

\def\Hr{{\rm H}}
\def\Zr{{\rm Z}}

\newcommand{\BRD}{{\rm BRD}}

\def\tg{{\mathfrak t}}

\newcommand{\cR}{{\mathcal R}}
\newcommand{\cS}{{\mathcal S}}
\newcommand{\cP}{{\mathcal P}}
\newcommand{\cD}{{\mathcal D}}
\newcommand{\cF}{{\mathcal F}}

\newcommand{\Tr}{{\rm Tr}}

\newcommand{\Gtz}{{G_t^{\prime\prime\, 0}}}
\newcommand{\Gza}{{G^{0,\ad}}}

\def\AAA{{\sf A}}
\def\DDD{{\sf D}}
\def\EEE{{\sf E}}
\def\XXX{{\sf X}}

\def\GmR{{\G_{{\rm m},\R}}}

\begin{document}

\title[Real two-step nilpotent Lie algebras]
{Real non-degenerate   two-step nilpotent\\ Lie algebras of dimension eight\\
 }

\author{Mikhail  Borovoi}
\address{Borovoi: Raymond and Beverly Sackler School of Mathematical Sciences,
Tel Aviv University, 6997801 Tel Aviv, Israel}
\email{borovoi@tauex.tau.ac.il}

\author{Bodgan Adrian Dina}
\address{Dina: Raymond and Beverly Sackler School of Mathematical Sciences,
Tel Aviv University, 6997801 Tel Aviv, Israel}
\email{bogdanadrian.dina@gmail.com}

\author{Willem A. de Graaf}
\address{De Graaf: Department of Mathematics, University of Trento, Povo (Trento), Italy}
\email{degraaf@science.unitn.it}

\thanks{ Borovoi and Dina were supported
by the Israel Science Foundation (grant 1030/22).}

\date{\today}

\keywords{Two-step nilpotent  Lie algebra, real form, real Galois cohomology}

\subjclass{Primary:
  15A21.  
Secondary:
  11E72  
, 20G05  
, 20G20.
}

\begin{abstract}
We  classify the non-degenerate two-step nilpotent Lie algebras of dimension 8 over the field of real numbers,
using  known results over complex numbers.
We  write explicit structure constants for these real Lie algebras.
\end{abstract}

\maketitle

\tableofcontents

\section{Introduction}

Classification lists of Lie algebras, in particular of nilpotent Lie algebras,
over the field of complex numbers $\C$  and over the field of real numbers $\R$,
appear to be an often used tool in mathematical physics; see, for instance, \cite{PSWZ} and \cite{PP}.
The problem of classification of nilpotent Lie algebras of arbitrary dimensions (even of
two-step nilpotent Lie algebras of arbitrary dimension with 3-dimensional center) is wild; see \cite{BLS05}.
However, it is possible to classify nilpotent Lie algebras in low dimensions.
Up to now, nilpotent Lie algebras over some fields have been classified up to dimension 7.
Lists of nilpotent Lie algebras of dimension at most 7 over different fields
can be found, in particular, in
\cite{ABG}, \cite{gong}, \cite{dG-6}, \cite{magnin}, \cite{Morozov}, \cite{Romdhani}.
There is no known classification of nilpotent algebras of dimensions greater than 7,
even over the field of complex numbers $\C$.

A {\em two-step nilpotent} Lie algebra over a field $\kk$
(synonyms: metabelian Lie algebra, nilpotent Lie algebra of class 2)
is a Lie algebra $L$ over $\kk$ such that
\begin{equation}\label{e:2-step}
[\hs[L,L],L]=0.
\end{equation}
Write $\z(L)$ for the center of $L$:
$$\z(L)=\{x\in L\ |\ [x,y]= 0\ \forall y\in L\}.$$
Set $L^{(1)}=[L,L]$.
Then condition \eqref{e:2-step} means that $L^{(1)}\subseteq \z(L)$.

Two-step nilpotent Lie algebras
form the first non-trivial subclass of nilpotent Lie algebras.
A classification of two-step nilpotent Lie algebras in dimensions up to 7
was given by Gauger \cite{Gauger} over an {\em algebraically closed} field
of characteristic different from 2,
and by Stroppel \cite{Stroppel} over an arbitrary field.

In \cite{GT} Galitski and Timashev introduced an invariant-theoretic approach
to classification of two-step nilpotent Lie algebras,
which allowed them to classify such Lie algebras over $\C$  up
to dimension 9 (in almost all cases).
They reduced the classification of two-step nilpotent Lie algebras
up to dimension 9 to  classification of orbits
of $\SL(m,\C)\times\SL(n,\C)$ in $\bigwedge^2 \C^m\otimes \C^n$ for $(m,n)$
taking values (5,4),  (6,3), and (7,2),
and they solved the classification problems for (5,4) and  (6,3) using the method
of $\theta$-groups due to Vinberg \cite{Vinberg76} and \cite{Vinberg79}.

Later, in the papers \cite{RZh8} and \cite{YD}, the
two-step nilpotent Lie algebras over $\C$ of dimension 8 were classified.
These results seem to be consistent with the results
of Galitski and Timashev \cite{GT}.

In the present paper, using the known classification of 8-dimensional
non-degenerate  two-step nilpotent  Lie algebras over $\C$
(due to Galitski and Timashev  \cite{GT}, and also to Ren and Zhu \cite{RZh8} and to Yan and Deng  \cite{YD})
we obtain a classification  over $\R$.
See the next section for the definition of a {\em non-degenerate}  two-step nilpotent Lie algebra.
We start with known results over $\C$ and use Galois cohomology.
We  compute the Galois cohomology
using the computer program \cite{dG-GalCohom} described in \cite{BG}.
Our main results are Tables 1--3.

We performed our computations using computational
algebra system {\sf GAP}; see \cite{GAP4}.
A small number of computations concerning automorphism groups of lattices
were performed on Magma \cite{Magma}.

The plan of our paper is as follows.
In Section \ref{s:reductions} we reduce our classification problem
to classification of orbits of the group $\GG(\R)=\GL(m,\R)\times \GL(n,\R)$
in the set of non-degenerate tensors $e\in\YY=(\bigwedge^2 \UU)^*\otimes \VV$
where $\UU=\R^m$, $\VV=\R^n$ for the pair $(m,n)$
taking values $(6,2)$, $(5,3)$, and $(4,4)$.
In Section \ref{s:Tables} we give the tables
of representatives of all orbits; this is our main result.

Section \ref{s:real-alg} contains preliminaries
on real algebraic groups and real Galois cohomology.
In Section \ref{s:BRD}, for a connected reductive complex algebraic group $G$,
we describe the action of the automorphism group $\Aut(G)$
on the canonical based root datum $\BRD(G)$.
Section \ref{s:theta} contains preliminaries on $\theta$-representations.
Starting Section \ref{s:Galois}, we compute our tables.
See below the idea of the computations.

In Section  \ref{s:Galois}, we consider a tensor
$e\in \YY$ and the stabilizer $\GG_e$ of $e$ in $\GG$.
We reduce the classification of the orbits of $\GG(\R)$ in $G\cdot e\subset Y$
(where $G=\GG(\C)$ and $Y=\YY\otimes_\R\C)$
to computing the Galois cohomology set $\Hr^1\GG_e$.

In order to compute $\Hr^1\GG_e$\hs, we embed $\g=\Lie \GG$ and $\YY$
into a $\Z$-graded real Lie algebra $\hat\g$
such that $\hat \g_0=\g$ and $\hat \g_1=Y$.
Moreover, we embed our real tensor $e$ into a real homogeneous $\ssl_2$-triple
$t=(e,h,f)$ with $h\in\hat\g_0$ and $f\in\hat\g_{-1}$.
Using $\hat g$ and $t$, we construct a reductive $\R$-subgroup
$\PP_t\subseteq\GG_e$ (not necessarily connected)
such that
\[\GG_e=R_u(\GG_e)\rtimes\PP_t\hs,\]
$R_u$ denoting the unipotent radical. Then by Sansuc's lemma we have
\[\Hr^1\GG_e=\Hr^1\PP_t\hs.\]

For computing $\Hr^1\PP_t$\hs,
we have a computer program \cite{dG-GalCohom}, described in \cite{BG},
which computes the Galois cohomology of a real algebraic group $\HH$,
and the input for which is a real basis of the Lie algebra $\Lie\HH$
and a set of representatives in $H=\HH(\C)$ of the component group $\pi_0(H)$.
Thus we need $\Lie\PP_t$ and $\pi_0(P_t)$.
It is easy to compute  $\Lie\PP_t$ using computer,
but computing $\pi_0(P_t)$ is tricky.
We computed $\pi_0(P_t)$ case by case via a computer-assisted calculation
with participation of a human mathematician.
For details see Sections \ref{s:Galois}-\ref{s:Details}.

In Appendix \ref{app:Duality}
we consider an alternative approach for the case $(m,n)=(4,4)$.
Namely, by duality (see Gauger \cite[Section 3]{Gauger}
or Galitski and Timashev \cite[Section 1.2]{GT})
our classification problem for  $(4,4)$ reduces to the
already solved classification problems for $(4,2)$ and $(3,2)$.
Our results for  $(4,4)$ are consistent with the results
of \cite{dG-6} for $(4,2)$ and $(3,2)$.

\subsection*{Notation}
In this paper, by an algebraic group we mean a {\em linear} algebraic group.
By letters $\GG,\HH,\dots$ in the boldface  font we denote {\em real} algebraic groups.
By the same letters, but in the usual (non-bold) font $G,H,\dots$,
we denote the corresponding {\em complex} algebraic groups
$G=\GG\times_\R\C$, $H=\HH\times_\R\C$,\dots (though the standard notations are $\GG_\C$ and $\HH_\C$)
and by the corresponding small Gothic letters $\g,\h,\dots$,
we denote the Lie algebras of $G,H,\dots$.
Similarly, for real vector spaces $\UU,\VV$,
we write $U= \UU\otimes_\R\C$, $V=\VV\otimes_\R\C$, \dots.

For a real algebraic group $\GG$, we denote by $\GG(\R)$ and $\GG(\C)$
the groups of the real points and the complex points of $\GG$, respectively;
see Section \ref{s:real-alg} for details.
By abuse of notation, we identify $G\coloneqq \GG\times_\R\C$
 with the group of $\C$-points $\GG(\C)=G(\C)$.
In particular, $g\in G$ means $g\in G(\C)$.

We gather some of our notations:
\begin{itemize}
\item $Z(G)$ denotes the center of an algebraic group $G$;
\item $\Aut(G)$ denotes the automorphism group of $G$;
\item $\Inn(G)$ denotes the group of inner automorphisms of $G$;
\item $\Out(G)=\Aut(G)/\Inn(G)$;
\item $\Lie(G)$ denotes  the Lie algebra of $G$;
\item $\Dyn(G)$ denotes the Dynkin diagram of a connected reductive group $G$;
\item $G^0$  denotes the identity component of an algebraic group $G$;
\item $\pi_0(G)=G/G^0$ denotes the component group of $G$;
\item $\Hr^1\GG=\Hr^1(\R,\GG)$, the first Galois cohomology of a real algebraic group $\GG$;
\item $\GL(n,\C)$ denotes the complex Lie group  of invertible complex $n\times n$-matrices,
and also the complex algebraic group with this group of $\C$-points;
\item $\GL(n,\R)$ is the real {\em Lie group} of invertible real $n\times n$-matrices;
\item $\GL_{n,\R}$ is the connected real {\em algebraic group} with the group of real points $\GL(n,\R)$.
\end{itemize}

\section{First reductions}
\label{s:reductions}

Let  $L$ be  a two-step nilpotent Lie algebra over a field $\kk$ of characteristic different from 2.
If $A$ is a nonzero abelian Lie algebra, consider the direct product $L\times A$,
which has $L\oplus A$ as the underlying vector space and has the commutator
\[ \big[ (x,a),(x',a')\big]=( [x,x'], 0)\quad\ \text{for}\ \,x,x'\in L,\ \, a,a'\in A.\]
This is again a two-step nilpotent Lie algebra;
we say that such a Lie algebra $L\times A$ is {\em degenerate}.
Clearly, in order to classify  two-step nilpotent Lie algebras of dimension $d$ over a field $\kk$,
it suffices to classify {\em non-degenerate}  two-step nilpotent Lie algebras over $\kk$ of dimension $\le d$.

\begin{lemma}
Let $L$ be a finite-dimensional   two-step nilpotent Lie algebra over a field $\kk$.
Then $L$ is non-degenerate if and only if $L^{(1)}= \z(L)$.
\end{lemma}
\begin{proof}
If $L$ is degenerate, that is,  $L=L'\times A$  where $A$ is a nontrivial abelian Lie algebra, then $\z(L)=\z(L')\oplus A$ and
\[L^{(1)}=(L')^{(1)}\subseteq \z(L')\,\subsetneqq\,\z(L')\oplus A\hs= \z(L),\]
whence $L^{(1)}\neq \z(L)$.

Conversely, assume that $L^{(1)}\neq \z(L)$. Choose a complement $A$ to  $L^{(1)}$ in $\z(L)$, so that
\[\z(L)=L^{(1)}\dotplus A\quad\ \text{with}\ \, A\neq 0,\]
and choose a complement $U'$ to $\z(L)$ in $L$, so that
\[L=U'\dotplus \z(L)=U'\dotplus L^{(1)}\dotplus A.\]
Here $\dotplus$ denotes internal direct sum.
Set
\[L'=U'\dotplus L^{(1)};\]
then $L=L'\dotplus A$.
If $x'\in L'\subset L$ and $y\in L$, then
\[ [x',y]\in L^{(1)}\subset L',\]
whence  $L'$ is an ideal of $L$. Clearly, $A$ is a central ideal of $L$ and
\[L=L'\times  A.\]
Thus $L$ is degenerate, as required.
\end{proof}

In this paper we classify {\em non-degenerate}
two-step nilpotent Lie algebras of dimension 8 over $\R$.
Clearly, classification of degenerate
two-step nilpotent Lie algebras of dimension 8 over $\R$
can be reduced to classification of non-degenerate
two-step nilpotent Lie algebras of smaller dimension over $\R$
(which is known).

Let  $L$ be a non-degenerate two step nilpotent Lie algebra over a field $\kk$ of characteristic different from 2.
Set $U=L/\z(L)$ and  $V=L^{(1)}\subseteq \z(L)$.
The Lie bracket in $L$ defines an skew-symmetric bilinear map
\begin{equation*}
\beta\colon U\times U\to V
\end{equation*}
and the induced linear map
\[\beta_*\colon\bigwedge^2 U\to V.\]

The triple $(U,V,\beta)$ is non-degenerate in the following sense:
the linear map $\beta_*$ is surjective, and for any nonzero $u\in U$,
there exists $u'\in U$ with $\beta(u,u')\neq 0$.

Let $L$ be a non-degenerate  two-step nilpotent Lie algebra.
Write $m=\dim U$, $n=\dim V$ where $U$,$V$ are as above.
Then $m+n=\dim L$ (because $L$ is non-degenerate).
We say then that $L$ is {\em of signature $(m,n)$}.

A non-degenerate two-step nilpotent Lie algebra $L$ of signature $(m,n)$
defines a non-degenerate triple $(U,V,\beta)$  of signature $(m,n)$
(that is, with $\dim U=m$, $\dim V=n$).
Conversely, a non-degenerate triple $(U,V,\beta)$ of signature $(m,n)$
defines a non-degenerate 2-step-nilpotent Lie algebra of signature $(m,n)$ with underlying vector space
\[L=U\oplus V\]
and with the Lie bracket
\[ [(u,v),(u',v')]=\big(0, \beta(u,u')\big)\quad\ \text{for}\ \,u,u'\in U,\ v,v'\in V.\]

We see that to classify non-degenerate
two-step nilpotent Lie algebras $L$ of signature $(m,n)$ up to an isomorphism
is the same as to classify non-degenerate triples $(U,V,\beta)$
with $\dim U=m$ and $\dim V=n$ up to isomorphism,
which in turn is equivalent to classification of the orbits of the Lie group
$\GL(m,\kk)\times \GL(n,\kk)$ in the set of non-degenerate skew-symmetric maps
\[ \beta\colon \kk^m\times\kk^m\to\kk^n.\]

We wish to classify non-degenerate skew-symmetric maps $\beta$ as above
over $\kk=\R$ for the signatures $(m,n)$
with $m+n=8$, that is,
\[ (1,7),\ \,(2,6),\ \,(3,5),\ \,(4,4),\ \,(5,3),\  \,(6,2),\  \,(7,1).\]
However, if $(m,n)=(7,1)$, then $\beta$ is a skew-symmetric bilinear form on $\R^7$,
but we know that there is no non-degenerate skew-symmetric bilinear forms on odd-dimensional spaces;
see, for instance,   Artin \cite[Theorem 3.7]{Artin} or Lang \cite[Theorem XV.8.1]{Lang}.
Moreover, if $m\le 3$, then $\dim\bigwedge^2 U=m(m-1)/2\le 3$, while $\dim V=n\ge 5$,
and therefore the linear map $\beta_*\colon \bigwedge^2 U\to V$ cannot be surjective.
Thus for $(m,n)=(1,7)$, $(2,6)$, $(3,5)$ there are no
non-degenerate skew-symmetric bilinear maps of those signatures.
It remains to classify the non-degenerate skew-symmetric maps $\beta$ for
\[ (m,n)\,=\,(4,4),\ \,(5,3),\ \, (6,2).\]

\newcommand{\gr}{\rowcolor[gray]{.89}}
\newcommand{\f}[1]{{$\scriptstyle{#1}$}}
\newcommand{\y}[1]{{\tiny{#1}}}
\newcommand{\ttt}{{\mathfrak t}}
\newcommand{\uuu}{{\mathfrak u}}
\newcommand{\sssp}{\mathfrak{\mathop{sp}}}

\newcommand{\cc}{{\uparrow}}
\newcommand{\ee}[3]{{e_{#1#2\cc#3}}}

\section{Tables}
\label{s:Tables}

In Tables 1--3 below,
we classify  the orbits of the group $\GG(\R)=\GL(m,\R)\times\GL(n,\R)$
acting on the set of {\em non-degenerate} skew-symmetric bilinear maps $ \R^m\times \R^m\to \R^n$ for $m+n=8$.
This corresponds to the isomorphism classes of non-degenerate two-step nilpotent real Lie algebras of dimension $8$.

In these tables, our real two-step nilpotent  Lie algebra $L$ is $\R^8$ with the standard basis $e_1,\dots,e_8$\hs.
The notations like  1 and 1-bis denote two real orbits contained in the same complex orbit.
The representatives $1,\,2,\,3\dots$ in each table were taken from \cite[Tables 2 and 8]{GT}.
Using Galois cohomology, we determined whether there are other orbits in the same complex orbit,
and if yes, we computed a representative of each real orbit.
It turned out that there are at most two real orbits in each complex orbit;
the other real orbit in the complex orbit containing the real orbit 1 is denoted by 1-bis.

In each row, the Lie bracket is given by the $(2,1)$-tensor $e$ given in the table,
as explained in Section \ref{s:reductions}.

For example, in the row 3-bis in Table \ref{tab:reps(5,3)}, our (2,1)-tensor is
\begin{equation}\label{e:3(2)}
 e=  \ee138 +\ee147 +\ee156 -\ee237 +\ee248 +\ee346
\end{equation}
where we write $\ee138$ for the $(2,1)$-tensor $(e_1^*\wedge e_3^*)\otimes e_8$\hs.
Here  $e^*_1$ and $e^*_3$ are basis vectors of the dual space
$\UU^*\coloneqq \Hom(\R^5,\R)$ with basis $e_1^*,\dots,e_5^*$,
and $e_8$ is a basis vector of the space $\VV=\R^3$ with basis $e_6,e_7,e_8$.
This tensor $e$ of formula  \eqref{e:3(2)} defines the following Lie bracket:
\begin{gather*}
 [e_1,e_3]=e_8\hs;\ \, [e_1,e_4]=e_7\hs;\ \, [e_1,e_5]=e_6\hs;\ \, [e_2,e_3]=-e_7\hs;\ \,
 [e_2,e_4]=e_8\hs;\ \,[e_3,e_4]=e_6\hs;\\
 [e_1,e_2]=0;\ \,[e_2,e_5]=0\hs;\ \,[e_3,e_5]=0\hs;\ \,[e_4,e_5]=0.
\end{gather*}

\newpage
\begin{longtable}{|l|l|l|l|l|l|l}
\caption{Nondegenerate real orbits of signature $(6,2)$}\label{tab:reps(6,2)}
\endfirsthead
\hline
\endhead
\hline
\endfoot
\endlastfoot
\hline
\y{No.}    &  {\Small Representative $e$ of an orbit} &\f{\pi_0} &\ \ \ \ \  \f{\g_t''}
        &\qquad{\Small  Rep.~in $U$} &\ \ {\Small Rep.~in $V$} \\
\hline
\gr
\y{1}
  &\f{\ee127 +\ee348 +\ee567 +\ee568} &\f{S_3}  &\f{3\hs\ssl(2,\R)}
  &\f{(1,0,0)+(0,1,0)+(0,0,1)} &\f{(0,0,0)+(0,0,0)}
\\
\gr
\y{1-bis}
  &\f{\ee127 +\ee348 -\ee367 +\ee457 -\ee568}
  &  &\f{\ssl(2,\C)+\ssl(2,\R)}
  &&
\\
\y2
  &\f{\ee148 +\ee157 +\ee238 +\ee467 } &\f1 &\f{2\hs\ssl(2,\R)+\t}
  &\f{(1,0)+(1,0)+(0,1)}    &\f{(0,0)+(0,0)}
\\
\gr
\y{3}
  &\f{ \ee148 +\ee157 +\ee238 +\ee267 +\ee347 } &\f1 &\f{\ssl(2,\R)+\t}
  &\f{(1)+(1)+(1)}    &\f{(0)+(0)}
\\
\y{4}
  &\f{ \ee137 +\ee168 +\ee248 +\ee257 }  &\f1  &\f{2\hs\ssl(2,\R)+\t}
  &\f{(1,1)+(0,1)}       &\f{(1,0)}
\\
\gr
\y5
  &\f{ \ee128 +\ee347 +\ee567 } &\f1    & \f{\sssp(4,\R)+\ssl(2,\R)+\t}
  &\f{(1,0,0)+(0,0,1)}    &\f{(0,0,0)+(0,0,0)}
\\
\y6
  &\f{ \ee128 +\ee167 +\ee257 +\ee347 } &\f1    & \f{2\hs\ssl(2,\R)+\t}
  &\f{(1,0)+(1,0)+(0,1)}    &\f{(0,0)+(0,0)}
\\
\hline
\end{longtable}

\begin{longtable}{|l|l|l|l|l|l|l}
\caption{Nondegenerate real orbits of signature $(5,3)$}\label{tab:reps(5,3)}
\endfirsthead
\hline
\endhead
\hline
\endfoot
\endlastfoot
\hline
\y{No.}    &\qquad{\Small Representative $e$ of an orbit} &\f{\pi_0} &\ \ \ \ \  \f{\g_t''}
        &\quad\ {\Small   Rep.~in $U$} &{\Small Rep.~in $V$}\\
\hline
\gr
\y{1}  &\f{ \ee126 +\ee158 +\ee238 +\ee257 +\ee347 +\ee456 }  &\f1   &\f{ \ssl(2,\R) }
       &\f{(4)}   &\f{(2)}
\\
\gr
\y{1-bis}    &\f{ \ee126 -\ee158 +2\ee236 +2\ee248 +\ee257 -2\ee347 -2\ee358 +2\ee456 }
        &   &\f{ \su(2) }&&
\\
\y{2}   &\f{ \ee146 +\ee157 +\ee238 +\ee247 +\ee356  }  &\f{C_2}   &\f{ 2\t }&&
\\
\y{2-bis}  &\f{\ee146 -\ee157 -\ee238 +\ee246 +\ee257 +\ee347 -\ee356 }
        &\f{}   &\f{ \t+\ug }&&
\\
\gr
\y{3}   &\f{ \ee148 +\ee156 +\ee237 +\ee248 +\ee346  }
        &\f{C_2}    &\f{ 2\t }&&
\\
\gr
\y{3-bis}   &\f{ \ee138 +\ee147 +\ee156 -\ee237 +\ee248 +\ee346 }
        &\f{}    &\f{ \t+\ug} &&
\\
\y4     &\f{ \ee128 +\ee157 +\ee237 +\ee256 +\ee346 }  &\f{1}    &\f{ 2\t } &&
\\
\gr
\y{5}   &\f{ \ee136 +\ee158 +\ee247 +\ee258  }   &\f{ S_3 }    &\f{ 3\t }  &&

\\
\gr
\y{5-bis}   &\f{\ee136 +\ee148 +\ee157 -\ee247 +\ee258}
&\f{ }    &\f{ 2\t+\ug }  &&
\\
\y6     &\f{ \ee138 +\ee157 +\ee237 +\ee246  }  &\f{1}    &\f{ 3\t }  &&
\\
\gr
\y7     &\f{ \ee138 +\ee147 +\ee156 +\ee237 +\ee246  }  &\f{1}    &\f{ 2\t } &&
\\
\y8     &\f{ \ee127 +\ee158 +\ee256 +\ee346  }  &\f{1}    &\f{ 2\hs\ssl(2,\R)+\t }
        &\f{(1,0)+(0,1)+(0,0)}      &\f{(0,1)+(0,0)}
\\
\gr
\y9     &\f{ \ee128 +\ee136 +\ee157 +\ee247 +\ee256  }  &\f{1}    &\f{ \ssl(2,\R)+\t }
        &\f{(2)+(1)}        &\f{(1,0)}
\\
\y{10}  &\f{ \ee128 +\ee146 +\ee237 +\ee356  }  &\f{1}    &\f{  \ssl(2,\R)+2\t }
        &\f{(1)+(1)+(0)}       &\f{(1)+(0)}
\\
\gr
\y{11}  &\f{ \ee147 +\ee158 +\ee236  }  &\f{1}    &\f{ 2\hs\ssl(2,\R)+2\t }
        &\f{(1,0)+(0,1)+(0,0)}      &\f{(1,0)+(0,0)}
\\
\y{12}  &\f{ \ee128 +\ee147 +\ee156 +\ee236 }  &\f{1}    &\f{ 3\t }  &&
\\
\hline
\end{longtable}

\begin{longtable}{|l|l|l|l|l|l|l}
\caption{Nondegenerate real orbits of signature $(4,4)$}\label{tab:reps(4,4)}
\endfirsthead
\hline
\endhead
\hline
\endfoot
\endlastfoot
\hline
\y{No.}    &\quad {\Small Representative $e$ of an orbit} &\f{\pi_0} &\ \ \ \ \  \f{\g_t''}
         &{\Small\   Rep.~in $U$} &\ {\Small Rep.~in $V$} \\
\hline
\gr
\y{1}   &\f{ \ee125 +\ee137 +\ee248 +\ee346 }    &\f{C_2} &\f{2\hs\ssl(2,\R)+\t}
        &\f{(1,0)+(0,1)}    &\f{(1,1)}
\\
\gr
\y{1-bis}  &\f{ \ee125 +\ee137 -\ee148 -\ee238 -\ee247 +\ee346 }
        &\f{} &\f{\ssl(2,\C)+\ug}
        &&
\\
\y{2}   &\f{ \ee128 +\ee135 +\ee147 +\ee237 +\ee246 }    &\f{1} &\f{ \ssl(2,\R)+\t }
        &\f{(1)+(1)}     &\f{(2)+(0)}
\\
\gr
\y{3}   &\f{ \ee127 +\ee138 +\ee145 +\ee236 }    &\f{1} &\f{ \ssl(2,\R)+2\t}
        &\f{(1)+(0)+(0)}       &\f{(1)+(0)+(0)}
\\
\hline
\end{longtable}

In the columns 3--6 ($\pi_0$, $\g_t''$, ``Rep.~in $U$'', ``Rep.~in $V$'')  of each of the tables,
we give certain invariants of the stabilizer $\GG_e$ of our tensor $e\in\Hom(\bigwedge^2 \UU,\VV)$
in the group $\GG=\GL(\UU)\times\GL(\VV)$.
We use these invariants  in order to compute the Galois cohomology of $\GG_e$,
which permits us to determine the real orbits in the complex orbit $G\cdot e$.

We define the invariant $\pi_0$ here: it is the component  group $\pi_0(G_e)$
of the stabilizer $G_e$ of our tensor $e$.
The real Lie algebra $\g_t''$ is defined in Section \ref{s:Galois},
and the representations in the columns ``Rep.~in $U$'' and  ``Rep.~in $V$''
are defined in  Section \ref{s:Details}.
We remark that the most tricky part of our calculations is the calculation of
$\pi_0(G_e)$; see Section \ref{s:Details} for an outline of the
methods that we have used.

We see from the table  that there are 27 isomorphism classes
of non-degenerate two-step nilpotent Lie algebras of dimension 8 over $\R$:
seven isomorphism classes of signature (6,2), sixteen isomorphism classes of signature (5,3),
and four isomorphism classes of signature (4,4).
Any isomorphism class over $\C$ comes from one or two isomorphism classes over $\R$.

\section{Real algebraic groups and real Galois cohomology}
\label{s:real-alg}

Let $\GG$ be a real linear algebraic group.
In the coordinate language, one may regard $\GG$ as
a subgroup in the general linear group $\GL(N,\C)$ (for some natural number $N$)
defined by polynomial equations with {\em real} coefficients in the matrix entries;
see Borel \cite[Section 1.1]{Borel-66}.
More conceptually, one may assume that $\GG$ is an affine group scheme
of finite type over $\R$; see  Milne \cite[Definition 1.1]{Milne-AG}.
With any of these two equivalent  definitions, $\GG$ defines a covariant functor
\begin{equation*}
A\mapsto \GG(A)
\end{equation*}
from the category of commutative unital $\R$-algebras to the category of groups.
We apply this functor to the $\R$-algebra $\R$ and obtain a real Lie group $\GG(\R)$.
We apply this functor  to the $\R$-algebra $\C$ and to the morphism of $\R$-algebras
\[\gamma\colon \C\to\C, \quad z\mapsto \bar z\quad\text{for }z\in \C,\]
and obtain a complex Lie group $\GG(\C)$ together with an anti-holomorphic involution
$\GG(\C)\to \GG(\C),$
which we denote by $\sigma_\GG$.
The Galois group $\Gamma$ naturally acts on $\GG(\C)$; namely,
the complex conjugation $\gamma$ acts by $\sigma_\GG$.
We have $\GG(\R)=\GG(\C)^\Gamma$ (the subgroup of fixed points).

We shall consider the linear algebraic group $G\coloneqq \GG\times_\R\C$
obtained from $\GG$ by extension of scalars from $\R$ to $\C$.
Since $G$ is an affine group scheme over $\C$, we have the ring of regular function
$\C[G]=\R[\GG]\otimes_\R\C.$
Our anti-holomorphic involution $\sigma_\GG$ of $\GG(\C)$
is {\em anti-regular} in the following sense:
when acting on the ring of holomorphic functions on $G$
(by acting by $\sigma_\GG^{-1}$ on the argument of a function,
and by complex conjugation on the value)
it preserves the subring $\C[G]$ of regular functions.
An anti-regular involution of $G$ is called also  a {\em real structure on} $G$.

\begin{remark}
If $G$ is a {\em reductive} algebraic group over $\C$ (not necessarily connected),
then any anti-holomorphic involution of $G$
is anti-regular. The hypothesis that $G$ is reductive is necessary.
For details and references see \cite[Section 1]{BT21}.
\end{remark}

A morphism of real linear algebraic groups $\GG\to\GG'$ induces a morphism of pairs
$(G,\sigma_\GG)\to (G',\sigma_{\GG'})$.
In this way we obtain a functor
$\GG\mapsto(G,\sigma_\GG)$.
By Galois descent this functor is an equivalence of categories;
for details and references see \cite[Section 1]{BT21} or \cite[Appendix A]{BG}.
In particular, any pair $(G,\sigma)$, where $G$ is
a complex linear algebraic group and $\sigma$ is a real structure on $G$,
is isomorphic to a pair coming from a real linear algebraic group $\GG$,
and any morphism of pairs $(G,\sigma)\to (G',\sigma')$ comes
from  a unique  morphism of the corresponding real algebraic groups.

When computing the Galois cohomology $\Hr^1\GG$ for a {\em real} algebraic group $\GG$,
we shall actually work with the pair $(G,\sigma)$,
where $G$ is a {\em complex} algebraic group and $\sigma$ is a real structure on $G$.
We shall shorten ``real linear algebraic group'' to ``$\R$-group''.

Let $\GG=(G,\sigma)$ be a real algebraic group (not necessarily connected or reductive).
The Galois group $\Gamma=\{1,\gamma\}$ acts on $G$ by
$$\upgam g=\sigma(g)\quad\ \text{for}\ \,g\in G.$$
We define the {\em first Galois cohomology set} $\Hr^1(\R,\GG)$ by
\[ \Hr^1(\R,\GG)=\Zr^1\GG/\sim.\]
Here $\Zr^1\GG=\{z\in G\ |\ g\cdot\upgam g=1\}$ is the {\em set of $1$-cocycles},
and two cocycles $z,z'\in \Zr^1\GG$ are {\em equivalent} (we write $z\sim z'$) if
$z=g^{-1}\cdot z'\cdot\upgam g$ for some  $g\in G$.
We shorten $\Hr^1(\R,\GG)$ to $\Hr^1\GG$.

For details see \cite[Section 3.3]{BGLa} or \cite[Section 4]{BG}.
See Serre's book \cite{Serre} for the Galois cohomology $H^1(\kk,\GG)$
for an algebraic group $\GG$ over an arbitrary field $\kk$.

\section{Action on the based root datum}
\label{s:BRD}

Let $G$ be a connected reductive group
over an algebraically closed field $\kk$.
Let $T\subset G$ be a maximal torus, and let $B\subset G$
be a Borel subgroup containing $T$.
We consider the {\em based root datum}
\[\BRD(G,T,B)=(X,X^\vee, \cR,\cR^\vee, \cS, \cS^\vee).\]
Here
\begin{itemize}
\item $X=X^*(T)$ is the character group of $T$;
\item $X^\vee=X_*(T)$ is the cocharacter group of $T$;
\item $\cR=\cR(G,T)\subset X$ is the root system;
\item $\cR^\vee=\cR^\vee(G,T)\subset X^\vee$ is the coroot system;
\item $\cS=\cS(G,T,B)\subset \cR$ is the system of simple roots;
\item $\cS^\vee=\cS^\vee(G,T,B)\subset \cR^\vee$ is the system of simple coroots.
\end{itemize}
For details see Springer \cite[Sections 1 and 2]{Springer}.

Recall that the root system $\cR$ is defined in term of the root decomposition
\[ \Lie G = \Lie T \oplus\bigoplus_{\alpha\in \cR} \g_\alpha\]
where $\g_\alpha$ is the eigenspace corresponding to the root $\alpha$.
For each $\alpha\in \cS$ we choose a nonzero element $x_\alpha\in \g_\alpha$\hs.
We write $\cP=\{x_\alpha\ |\ \alpha \in\cS\}$
and say that $\cP$ is a {\em pinning} of $(G,T,B)$.

We write $\cS=\{\alpha_1,\dots\alpha_r\}$ and consider the Cartan matrix
with entries $a_{ij}=\langle \alpha_i,\alpha_j^\vee\rangle$.
Recall that the {\em Dynkin diagram} $\Dyn(G)$ is the graph
whose set of vertices is the set of simple roots $\cS$
and whose set of edges is defined in the usual way
using the Cartan matrix; see, for instance, \cite[Section 3.1.7]{GOV}.

We say that $(T,B,\cP)$ is a Borel triple in $G$.
It is well known that if $(T',B',\cP')$ is another Borel triple,
then there exists a unique element $g^\ad=gZ(G)\in \Inn(G)\coloneqq G/Z(G)$
such that $gTg^{-1}=T'$, $gBg^{-1}=B'$, $g\cP g^{-1}=\cP$.
This element $g^\ad$ induces an isomorphism
$g^{\ad*}\colon\BRD(G,T',B')\isoto \BRD(G,T,B)$.
Moreover, this induced isomorphism $g^{\ad*}$ does not depend
on the choice of the pinning $\cP$ as above.
Thus for given $G$ we can canonically identify
the based root data $\BRD(G,T,B)$
for all Borel pairs $(T,B)$.
We obtain the canonical based root datum $\BRD(G)$.

The automorphism group $\Aut(G)$ naturally acts on  $\BRD(G)$,
and so we obtain a canonical homomorphism
\[\phi\colon \Aut(G)\to\Aut\, \BRD(G).\]
We describe $\phi$. Choose a pinning $\cP=(x_\alpha)$ of $(G,B,T)$.
Write $\BRD(G)=\BRD(G,T,B)$.
Consider the Borel triple $(T,B,\cP)$.
Let $a\in\Aut(G)$.
Then $\big(a(T),a(B),a(\cP)\big)$ is again a Borel triple in $G$,
and therefore there exists $g_a\in G$
such that
\[g_a\cdot a(T)\cdot g_a^{-1}=T,\quad\ g_a\cdot a(B)\cdot g_a^{-1}=B,\quad\ g_a\cP g_a^{-1}=\cP.\]
We see that the automorphism $\inn(g_a)\circ a$ of $G$
preserves the Borel triple $(T,B,\cP)$ and thus induces an automorphism
$\phi(a)$ of $\BRD(G,T,B)$.
One checks that the obtained automorphism $\phi(a)$ does not depend on the choice of $\cP$ and  $g_a$ as above.
For details see \cite[Section 3]{BKLR}.
By construction, the subgroup $\Inn(G)\subseteq\Aut(G)$ acts on $\BRD(G)$ trivially,
and so we obtain an action of $\Out(G)\coloneqq\Aut(G)/\Inn(G)$ on $\BRD(G)$.
The action of $\Out(G)$ on $\BRD(G)$, in particular, on $\cS$ and $\cS^\vee$, induces an action on $\Dyn (G)$.

We embed $X^\vee$ into $\tg\coloneqq \Lie T$ as follows.
Let $\nu\in X^\vee$, $\nu\colon \kk^\times\to T$.
Consider $d\nu\colon \kk\to \tg$ and set $h_\nu=(d\nu)(1)\in\tg$.

Consider the center $Z(G)$, its identity component $Z(G)^0$
(which is a torus), and the cocharacter group
$X_Z^\vee= X_*(Z(G)^0)$.
We can identify
\[X^\vee_Z=\{\nu\in X^\vee\ |\ \langle\alpha,\nu\rangle=0\ \,\text{for all}\ \alpha\in \cS\}.\]
The group $\Out(G)$ naturally acts on the torus $Z(G)^0$ and on its cocharacter group $X^\vee_Z$.
Moreover, it acts on the Lie algebra $\z\coloneqq\Lie Z(G)$ and on the lattice
$\{h_\nu\in\z\ |\ \nu\in X_Z^\vee\}.$

For  $\alpha\in \cS\subset X$, we consider
$\alpha^\vee\in \cS^\vee\subset X^\vee$,
and by abuse of notation we write $h_\alpha$ for $h_{\alpha^\vee}\in\tg$.
The set $\{h_\alpha\,|\,\alpha\in \cS\}$ is a basis
of the  Lie algebra $\tg\cap [\g,\g]$
where $[\g,\g]$ is the derived subalgebra of $\g$.

For each $\alpha\in\cS$ we choose a nonzero vector $x_\alpha\in\g_\alpha$.
Then we have
$\Ad(t) x_\alpha=\alpha(t)\cdot x_\alpha$ for $t\in T$,
whence
\[[h_\beta,x_\alpha]=(d\alpha)(h_{\beta^\vee})\cdot x_\alpha
=(d\alpha)(d\beta^\vee(1))\cdot  x_\alpha=d(\alpha\circ\beta^\vee)(1)\cdot x_\alpha
           =\langle\alpha,\beta^\vee\rangle\cdot x_\alpha\]
because $(\alpha\circ\beta^\vee)(t)=t^{\langle\alpha,\beta^\vee\rangle}$ for $t\in\C^\times$.

We choose $y_\alpha\in\g_{-\alpha}$ such that  $[x_\alpha\hs, y_\alpha ]=h_\alpha$\hs.
Then $[h_\alpha,y_\alpha]=-2y_\alpha$.
Note that the set $\{x_\alpha,y_\alpha\,|\, \alpha\in \cS\}$
generates the subalgebra  $[\g,\g]$ of $\g$.

The set $h_\alpha,x_\alpha,y_\alpha$ for $\alpha\in\cS$
has the following properties (see \cite[Section 2.9.3, formulas (2.1)\hs]{dG}):
\begin{align*}
\begin{aligned}
&[h_\alpha\hs,h_\beta]=0;\\
&[x_\alpha\hs,y_\beta]=\delta_{\alpha\beta} h_\alpha;\\
&[h_\beta\hs, x_\alpha]=\langle \alpha,\beta^\vee\rangle x_\alpha;\\
&[h_\beta\hs, y_\alpha]=-\langle \alpha,\beta^\vee\rangle y_\alpha\quad\ \text{for}\ \,\alpha,\beta\in \cS.
\end{aligned}
\end{align*}
We say that  $h_\alpha,x_\alpha,y_\alpha$ is a canonical generating set for $[\g,\g]$

Now let $G$ be a reductive group, {\em not necessarily connected,}
over an algebraically closed field $\kk$.
We consider the action of $G$ on the Dynkin diagram $\Dyn G^0\coloneqq \Dyn \Gza$,
where $\Gza=(G^0)^\ad\coloneqq G^0/Z(G^0)$ denotes the adjoint  group corresponding to
the identity component $G^0$ of the reductive group $G$.

As above, we choose a maximal torus $T\subset \Gza$
and a Borel subgroup $B\subset \Gza$ containing $T$.
We consider the based root datum
\[\BRD\big( \Gza,T,B\big)=(X,X^\vee, \cR,\cR^\vee, \cS, \cS^\vee).\]
For each simple root $\alpha\in \cS$  we choose canonical generators
$x_\alpha,y_\alpha, h_\alpha \in \Lie \Gza=[\g,\g]$ as explained in Section \ref{s:theta}, where $\g=\Lie G$.
Observe that the set $\{x_\alpha,y_\alpha\,|\, \alpha\in \cS\}$
generates the semisimple Lie algebra $[\g,\g]$.

Consider the action of $G$ by conjugation on $G^0$ and on $\Gza$.
We obtain a homomorphism
\[ G\to \Aut \Gza.\]

\begin{lemma}
Consider the homomorphism
\[\phi_{\Dyn}\colon G\to \Aut(\Dyn G^0)=\Aut(\Dyn \Gza)\]
and the group
\[A_1=\{g\in G\ |\ \Ad(g) x_\alpha=x_\alpha,\,
   \Ad(g) y_\alpha=y_\alpha\ \forall \alpha\in \cS.\} \]
Then
$ \ker \phi_{\Dyn}=G^0\cdot A_1$.
\end{lemma}

\begin{corollary}
\label{c:A1-Dyn}
We have $A_1\subset \ker\phi_{\Dyn}$, and
the homomorphism $\pi_0(A_1)\to \pi_0(\ker\phi_{\Dyn})$
induced by the inclusion homomorphism $A_1\into \ker\phi_{\Dyn}$
is surjective.
\end{corollary}

This corollary gives us a method to compute $\pi_0(\ker\phi_{\Dyn})$.

\begin{proof}[Proof of the lemma]
Consider the canonical homomorphism
\[\psi\colon G\to \Aut\Gza= \Aut(\Lie \Gza).\]
Since the elements $x_\alpha,\, y_\alpha$ generate $[\g,\g]=\Lie \Gza$,
we have $A_1=\ker\psi$.
On the other hand,
we can factor the homomorphism $\phi_{\Dyn}$ as follows:
\[ \phi_{\Dyn}\colon G\labelto\psi \Aut \Gza\to
    \Aut \BRD\big( \Gza,T,B\big)=\Aut\!\big(\Dyn \Gza\big).\]
It follows that $A_1=\ker\psi\subseteq \ker\phi_{\Dyn}$\hs.
Writing $\Out\Gza=\Aut\Gza/\Inn\Gza$, we have that the homomorphism
$$\Aut \Gza\to\Aut \BRD\big( \Gza,T,B\big)$$
 induces an isomorphism
$$\Out \Gza\isoto\Aut \BRD\big( \Gza,T,B\big);$$
see Conrad \cite[Proposition 1.5.5]{Conrad}.
It follows that $\ker\phi_{\Dyn}$ is the preimage in $G$
of the subgroup $\Gza=\Inn \Gza\subseteq\Aut\Gza$.
Since $\Gza$ is the identity component of $\Aut\Gza$,
we have $\psi(G^0)=\Gza$, and hence for any $g\in\ker\phi_{\Dyn}$
there exists $g_0\in G^0$ such that $\psi(g_0g)=1\in \Aut \Gza=\Aut\,[\g,\g]$,
that is,  $g_0g\in\ker\psi= A_1$, as required.
\end{proof}

\section{Vinberg's $\theta$-representations}
\label{s:theta}

Vinberg \cite{Vinberg76}, \cite{Vinberg79} introduced a class of
representations of algebraic groups which share many properties with the
adjoint representation of a semisimple Lie algebra. This makes it possible
to classify the orbits corresponding to such a representation. These
representations are constructed from a $\Z/m\Z$-grading or a $\Z$-grading of
a split semisimple Lie algebra over a field $\kk$.  In this paper we deal exclusively with
$\Z$-gradings of split semisimple real Lie algebras.

Let $\g$ be a split semisimple Lie algebra over $\R$.
We construct a class of $\Z$-gradings of $\g$. Fix a split Cartan
subalgebra $\t$ of $\g$ with corresponding root system $\cR=\cR(\g,\t)$. For $\alpha\in
\cR$, we denote the corresponding root space by $\g_\alpha$. Fix a basis
of simple roots $\cS=\{\alpha_1,\ldots,\alpha_\ell\}$ of $\cR$. Let $(d_1,\ldots,
d_\ell)$ be a sequence of non-negative integers. We define a function
$$
d\colon \cR\to \Z,\quad\ d\Big(\sum_{i=1}^\ell m_i \alpha_i \Big) = \sum_{i=1}^\ell m_i d_i\hs.
$$
We let $\g_0$ be the subspace of $\g$ spanned by $\t$ and all subspaces $\g_\alpha$ with
$d(\alpha)=0$. Furthermore, for $i\in \Z$, $i\neq 0$, we let $\g_i$ be
the subspace of $\g$ spanned by all $\g_\alpha$ with $d(\alpha)=i$. Then
$$\g = \bigoplus_{i\in \Z} \g_i$$
is a $\Z$-grading of $\g$.
We observe that $\g_0$ is a reductive Lie subalgebra; see \cite[Theorem 8.3.1]{dG}.
If
$$\alpha=m_1\alpha_1+\dots+m_\ell\alpha_\ell\in\cR
        \quad\ \text{with}\ \, m_i\in\Z_{\ge 0}$$
is a positive root with $m_{i}>0$ for some $i$ such that $d_{i}>0$,
then $d(\alpha)\ge m_{i} d_{i}>0$ and hence $\g_{\alpha_i}\cap\g_0=0$.
It follows that the simple roots  $\alpha_i$ such
that $d_i=0$ form a basis of the root system of $\g_0$ with respect to $\t$.

Let $\GG$ be the inner automorphism group of $\g$. This algebraic $\R$-group is also known as
the adjoint group of $\g$, or as the identity component of the automorphism
group of $\g$. For $x\in \g$ we denote its adjoint map by $\ad x$, so
$(\ad x)(y)= [x,y]$. The Lie algebra of $\GG$ is $\ad \g = \{ \ad x \mid x\in \g\}$.
Consider the Lie subalgebra $\ad \g_0 = \{ \ad x \mid x\in \g_0\}$.
Let $\TT\subset G$ be the Cartan subgroup (maximal torus) of $\GG$ with Lie algebra $\ad\t$,
and let $\GG_0$ denote the algebraic subgroup of $\GG$ generated by $\TT$
and the elements $\exp (\ad x)$ for $x\in \g_\alpha$ with $d(\alpha)=0$;
then $\GG_0$ is the connected real algebraic subgroup of $\GG$
whose Lie algebra is $\ad \g_0$. Since
$[\g_0,\g_i]\subset \g_i$\hs, the action of $\GG_0$ on $\g_1$ leaves the
subspaces $\g_i$ invariant. The representation of $\GG_0$ in $\g_1$ so obtained
is called a $\theta$-representation, and $\GG_0$ together with its action on
$\g_1$ is called a $\theta$-group.

For $x\in \g_1$ and $y\in \g_i$\hs, we have $(\ad x)^k(y) \in \g_{i+k}$.
Since the grading is a $\Z$-grading, we have $\g_{i+k}=0$ for sufficiently big $k$,
implying that $\ad x$ is nilpotent.
It can be shown that each real element $x\in \g_1$ lies in a real
{\em homogeneous $\ssl_2$-triple} $(h,x,y)$, where $h\in \g_0$ and $y\in \g_{-1}$ and
\begin{equation*}
[h,x]=2x,\, [h,y]=-2y,\, [x,y]=h;
\end{equation*}
see \cite[Lemma 8.3.5]{dG}.
The proof of \cite[Lemma 8.3.5]{dG} shows how to construct a homogeneous
$\ssl_2$-triple $(h,x,y)$ for a complex element $x$. When applied to a real
element $x$, this method gives a real $\ssl_2$-triple.

Consider two elements $x,x'\in \g_1$ lying in the homogeneous $\ssl_2$-triples $(h,x,y)$,
$(h',x',y')$. Then $x,x'$ are $G_0$-conjugate if and only if the two
triples are $G_0$-conjugate if and only if $h,h'$ are $G_0$-conjugate; see
\cite[Theorem 8.3.6]{dG}. It is possible to devise an algorithm for
classifying the $G_0$-orbits in $\g_1$ based on this fact; see
\cite[Section 8.4.1]{dG}. An alternative method for this is based on
Vinberg's theory of {\em carrier algebras}; see \cite{Vinberg79}. Since
a $\theta$-group for a $\Z$-grading or a $\Z/m\Z$-grading
has a finite number of {\em nilpotent} orbits
in $\g_{1,\C}$
(see \cite[Corolary 8.3.8]{dG}),
and in the case of a $\Z$-grading all $G_0$-orbits in $\g_1$ are nilpotent,
for  a $\Z$-grading there are finitely many  $G_0$-orbits in $\g_{1,\C}$.

Now we describe the $\Z$-gradings that are relevant to this
paper. They are constructed using a sequence $(d_1,\ldots,d_\ell)$ where
precisely one of the $d_i$ is 1, and all others are 0. We give such a
grading by giving the Dynkin diagram of $\g$ where the node corresponding to
$\alpha_i$ with $d_i=1$ is colored black. We consider the following three
$\Z$-gradings:

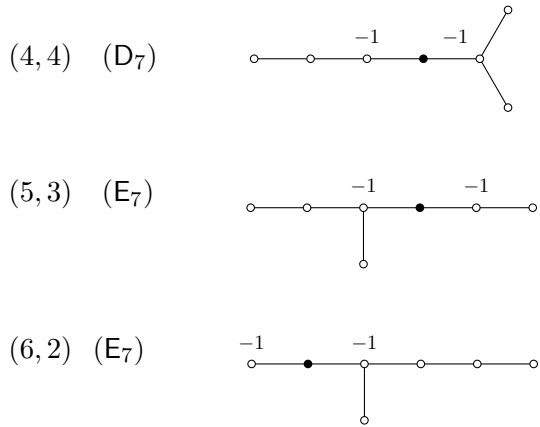
\tikzset{/Dynkin diagram,mark=o,
edge length=0.75cm,arrow shape/.style={-{angle 90}}}

\begin{align*}
& (4,4)\quad (\DDD_7)\qquad\quad
\dynkin[%
arrow shape/.style={-{angle 90}},
labels={},
labels*={,,-1,,-1}
]
D{ooo*ooo}
\\
\\
&  (5,3)\quad  (\EEE_7)\qquad\quad
\dynkin[%
upside down,
labels={,,,-1,,-1},
]
E{oooo*oo}
\\
\\
&  (6,2)\quad\hskip-1.2mm   (\EEE_7)\qquad\quad
\dynkin[%
upside down,
labels={-1,,,-1},
]
E{oo*oooo}
\end{align*}

Consider such a $\Z$-grading labeled $(m,n)$
with $(m,n)$ taking values $(4,4),\ (5,3),\ (6,2)$.
Write $\GG_0^\der=[\GG_0\hs,\GG_0]$, the derived subgroup of $\GG_0$.
We see from the Dynkin diagram that the semisimple group $\GG_0^\der$
is of type $\AAA_{m-1}\times \AAA_{n-1}$.
Therefore, the universal cover $\wt\GG_0$ of $\GG_0^\der$
is a split simply connected semisimple group of type $\AAA_{m-1}\times \AAA_{n-1}$
and hence it is isomorphic to $\SL_{m,\R}\times\SL_{n,\R}$; we fix this isomorphism.
We obtain an isogeny (surjective homomorphism with finite kernel)
$$\SL_{m,\R}\times\SL_{n,\R}\,\isoto\, \wt\GG_0\onto \GG_0^\der\hs.$$
Of course, to specify this isogeny we need to work with Lie algebras.

Since $\GG$ is of adjoint type, the set of simple roots
$\{\alpha_1,\dots,\alpha_\ell\}$ (with $\ell =7$)
is a basis of the character group $\XXX^*(\TT)$.
Write $\TT_0=\TT\cap \GG_0$;
then the set $\{\alpha_1,\dots,\alpha_\ell\}\smallsetminus\{\alpha_{i_0}\}$
is a basis of $\XXX^*(\TT_0)$, and $\TT=\TT_0\times \TT^1$,
where $\TT^1=Z(\GG_0)\cong\GmR$ is a 1-dimensional split $\R$-torus
with character group $\XXX^*(\TT^1)=\Z\cdot\alpha_{i_0}$\hs.
Write
\[\GG_{m,n}=\SL_{m,\R}\times\SL_{n,\R}\times \GmR\hs;\]
we obtain an isogeny $\GG_{m,n}\onto \GG_0$\hs.
Since $\g_1$ is the direct sum of eigenspaces $\g_\alpha$ of $\TT$ in $\g$
where $\alpha$ runs over the roots
\[\alpha=m_1\alpha_1+\dots+m_\ell\alpha_\ell\quad\ \text{with}\ \,m_{i_0}=1,\]
we see that $t\in\TT^1(\C)=\C^\times$ acts on $\g_1$ by multiplication by $t$.

We compute the representation of $\TT_0^\der$ in $\g_1$.
For $1\leq i \leq \ell=7$,
let
$$x_i\in \g_{\alpha_i},\quad y_i\in \g_{-\alpha_i}, \quad h_i\in \t$$
be {\em canonical generators} of $\g$.
For the root systems of types $\DDD_7$ and $\EEE_7$, this means that each $(h_i,x_i,y_i)$ is an
$\ssl_2$-triple, that for $i\neq j$ we have
$[h_i,x_j] = -x_j$, $[h_i,y_j] =y_j$ if $i,j$ are connected
in the Dynkin diagram, otherwise $[h_i,x_j]=[h_i,y_j]=0$, and that $[x_i,y_j] =0$.
Let $i_0$ be such that $d_{i_0}=1$. Then the elements $h_i$, $x_i$,
$y_i$ for $i\neq i_0$ form a canonical generating set of the semisimple
part  $\g_0^\der=[\g_0,\g_0]$ of $\g_0$.
In the three cases above it is easy to check
that $\alpha_{i_0}$ is the unique lowest weight of $\g_1$.
This means  that
  (1) for every $i\neq i_0$ we have that $\alpha_{i_0}-\alpha_i$
is not a root, and
  (2) for every root $\alpha\neq\alpha_{i_0}$ with $d(\alpha)=1$,
there is  $i\neq i_0$ such that $\alpha-\alpha_i$ is a root.
Here (1) is obvious, because any root must be
a linear combination of simple roots with non-negative coefficients, or
a linear combination of simple roots with non-positive coefficients.
Assertion (2) can be checked using a list of the positive roots
of the root systems of type $\DDD_7$ and $\EEE_7$\hs, respectively.
(See \cite[Section 8]{Vinberg76} for a more
conceptual approach). The positive roots can be listed using a
computer program like {\sf GAP}; alternatively, they can be found in
Bourbaki \cite{Bourbaki}.
Hence $x_{i_0}\in \g_1$ is the unique (up to scalar) lowest weight vector
(here ``lowest weight vector" means that $[y_i,x_{i_0}]=0$ for $i\neq i_0$).
Furthermore, $[h_i,x_{i_0}] = -x_{i_0}$ if $i$ and $i_0$ are connected in the Dynkin
diagram, and $[h_i,x_{i_0}] = 0$ otherwise. Hence the $\g_0$-module
$\g_1$ is irreducible and we know its lowest weight $\lambda_l$;
the numerical labels $\langle\lambda_l,\alpha_i^\vee\rangle$
are given on the Dynkin diagrams above.
From this we can determine the highest weight of $\g_1$.
Namely, let $\tau$ be the only nontrivial automorphism
of the Dynkin diagram $\AAA_\ell$ when $\ell>1$,
and the trivial automorphism of $\AAA_\ell$ for $\ell=1$.
Then the highest weight $\lambda_h$ is given by $\lambda_h=-\tau(\lambda_l)$;
see \cite[Section 3.2.6, Proposition 2.3 and Theorem 2.13]{GOV}.
We see that the numerical labels of $\lambda_h$ are 1 at the node neighboring an extreme node of $\AAA_m$\hs,
and 1 at an extreme node of $\AAA_n$ for $m\ge n$.
This means that there exists an isomorphism $\wt\GG_0\isoto \SL_{m,\R}\times \SL_{n,\R}$ with $m\ge n$ such that
$\g_1$ is isomorphic to $(\bigwedge^2 \R^m)^* \otimes \R^n$ as a $\GG_{m,n}$-module;
see \cite[Table 5]{OV}.
An element $(a_m,b_n,c_1)\in G_{m,n}$ acts  on $(\bigwedge^2 \C^m)^* \otimes \C^n$ as follows:
$$
(a_m,b_n,c_1)\cdot (\phi\otimes v) = c_1(a_m\cdot \phi)\otimes (b_n\cdot v)
\quad\ \text{for}\ \, \phi\in \big(\bigwedge\nolimits^2 \C^m\big)^*,\ v\in\C^n.
$$
We note that we have a canonical isomorphism
$(\bigwedge^2 \R^m)^* \otimes \R^n \cong \Hom( \bigwedge^2 \R^m,\R^n)$.

\section{Using Galois cohomology}
\label{s:Galois}

Set
\[\YY=\Hom\big(\bigwedge^2 \UU,\hs \VV\big)=\Hom\big(\bigwedge^2 \R^m,\hs \R^n\big), \quad\
Y\coloneqq \YY\otimes_\R\C=\Hom\big(\bigwedge^2 \C^m,\hs \C^n\big)\]
with the standard action of the Galois group $\Gal(\C/\R)$ on $Y$.
Write
\begin{align*}
&\GG=\GL_{m,\R}\times\GL_{n,\R}\hs,\quad\ G=\GL(m,\C)\times \GL(n,\C), \\
&\GG'=\SL_{m,\R}\times\GL_{n,\R}\hs,\quad\ G'=\SL(m,\C)\times \GL(n,\C),\\
&\GG''\coloneqq\GG_{m,n}=\SL_{m,\R}\times\SL_{n,\R}\times\G_{m,\R}\hs,
      \quad\ G''\coloneqq G_{m,n}=\SL(m,\C)\times \SL(n,\C)\times \C^\times.
\end{align*}
See Notation in the Introduction for the notations $\GL_{m,\R}$ and $\GL(m,\C)$.
The group $\GG$ acts on $\YY$ in the standard way, and we have a composite homomorphism
\[ \GG''\labelto p \GG'\into \GG.\]
Here the surjective homomorphism $p$ is given by
\begin{equation}\label{e:p}
p \colon \GG''\to \GG',\ \  (a_m,b_n,c_1)\mapsto (a_m, c_1 b_n)
\end{equation}
where  $a_m\in \SL(m,\C)$, $b_n\in\SL(n,\C)$, $c_1\in\C^\times$,
and we write $c_1 b_n$ for the product of the scalar $c_1$ and the matrix $b_n$\hs. Then
\[\ker p=\big\{(I_m,\hs\lambda^{-1} I_n,\hs\lambda)\ \big|\ \lambda^n=1\big\}\cong\mu_n\hs,\]
where $\mu_n$ denotes the group of $n$-th roots of unity.

Let $e\in \YY=\Hom\big(\bigwedge^2 \R^m, \R^n\big)$, which we view
as embedded into $Y\coloneqq\YY\otimes_\R\C$.
As seen in Section \ref{s:theta}, $\Lie G''$ is
isomorphic to $\hat \g_0$, the zero component of a $\Z$-graded simple complex
Lie algebra $\hat\g$. Moreover, $Y$ is isomorphic as a $G''$-module to
$\hat \g_1$  (in Section \ref{s:theta} it was denoted by $\g_1$).
We use the machinery of $\ssl_2$-triples, as indicated in
Section \ref{s:theta}, to classify the $G''$-orbits in $Y$.
Let $t=(h,e,f)$ be a real homogeneous $\ssl_2$-triple containing $e$, where
$h\in \hat\g_0$, $f\in\hat\g_{-1}$.
We write
\[ \GG_e=\Stab_{\GG}(e),\quad \ \GG'_e=\Stab_{\GG'}(e)
         \quad\ \GG''_e=\Stab_{\GG''}(e),\quad\  \GG''_t=\Stab_{\GG''}(t) .\]

Let $e\in \YY$. We denote by $\g_t''$  the centralizer (stabilizer)
of a real $\ssl_2$-triple  $t=(h,e,f)$ in the real Lie algebra
$\g''=\Lie \GG''= \ssl(m,\R)\times \ssl(n,\R)\times \R$.
These stabilizers for the real orbits are tabulated in fourth column of each of the  Tables 1--3, where
by $\t$ we denote the  Lie algebra of a one-dimensional {\em split} torus of $\GG$,
and by $\ug$ we denote the Lie algebra of a one-dimensional {\em compact} torus of $\GG$.

\begin{lemma}\label{l:using}
Assume that $e\in \YY$, and consider $\GG_e$, which is a real algebraic group.
Then there is a canonical bijection between $\Hr^1 \GG_e$ and the set of
real orbits (the orbits of $\GG(\R)$ in $\YY$) contained in the complex orbit $G\cdot e$.
\end{lemma}

\begin{proof}
By Serre \cite[Section I.5.4, Proposition 36]{Serre},
see also \cite[Proposition 3.6.5]{BGLa},
we have a canonical bijection between
$\ker[\Hr^1 \GG_e\to\Hr^1 \GG]$
and the set of orbits of $\GG(\R)$ in $G\cdot e\cap \YY$.
Moreover, since $\GG=\GL_{m,\R}\times\GL_{n,\R}$,
we have $\Hr^1 \GG=1$ (see Serre \cite[Section X.1, Proposition 3]{SerreLF},
and the lemma follows.
\end{proof}

We specify the map of the lemma. Recall that $\Hr^1 \GG_e=\Zr^1 \GG_e/\!\sim$
where $\Zr^1 \GG_e$ is the set of 1-cocycles; see Section \ref{s:real-alg}.
Here we write $z\sim z'$ if there exists $g\in G_e$ such that \,$z= g^{-1}z'\,\upgam g$.
Let $[z]\in \Hr^1 \GG_e$ be the cohomology class represented by a cocycle $z$.
Since $\Hr^1 \GG=1$,
there exists $g\in G$ such that
\[
z=g^{-1}\cdot 1\cdot\upgam g,\quad\text{that is,}\quad \upgam g=gz.
\]
One can easily find such an element $g$ using computer or by hand.
We have
\[\upgam(g\cdot e)=\upgam g\cdot  e=g\hs z\cdot e=g\cdot e,\]
because $z\in G_e$. Thus $g\cdot e$ is real (contained in $\YY$),
and to $[z]$ we assign the $\GG(\R)$-orbit of $g\cdot e$.

\section{The stabilizer of $e$ and the centralizer of $t=(h,e,f)$}
\label{s:e-t}

By Lemma \ref{l:using} we can use $\Hr^1 \GG_e$
in order to classify orbits of $\GG(\R)$ in $\YY$.
Therefore, we need to compute $\Hr^1\GG_e$.
To compute $\Hr^1 \GG_e$\hs, we embed $e$ into am $\ssl_2$-triple $t=(e,h,f)$
as in Section \ref{s:Galois}.
We compute $\GG''_t$ using the theory of $\theta$-groups.
Below we describe some  relations between $\GG''_t$ and $\GG_e$\hs.

 By \cite[Theorem 4.3.16]{BGLa} we have
\begin{equation}\label{e:Le}
\GG''_e=R_u(\GG''_e)\rtimes \GG''_t
\end{equation}
where $R_u$ denotes the unipotent radical.
By Sansuc's lemma
(\cite[Lemma 1.13]{Sansuc}, see also \cite[Proposition 3.2]{BDR} and \cite[Proposition 7.1]{BG})
the inclusion $\GG''_t\into \GG''_e$ induces a bijection  $\Hr^1\hs \GG''_t=\Hr^1\hs \GG''_e$\hs.

\begin{lemma}\label{l:ker-p}
For the homomorphism $p$ of \eqref{e:p}, we have $\ker p\subseteq Z(G''_t)$
where $ Z(G''_t)$ denotes the center of $G''_t$\hs.
\end{lemma}

\begin{proof}
Clearly, $\ker p\subset G''_e$\hs. Since $\ker p \subseteq Z(G'')$,
we have $\ker p\subseteq Z(G''_e)$.
Each element $x\in \ker p$ is semisimple, and in view of \eqref{e:Le}
there exists $g\in G''_e$ such that $gxg^{-1}\in G''_t$;
see Hochschild \cite[Theorem VIII.4.3]{Hochschild}.
Since $x\in Z(G''_e)$, we see that $x\in G''_t$
and  that $x\in Z(G''_t)$\hs, as required.
\end{proof}

Since the homomorphism $p$ is surjective, from \eqref{e:Le} we obtain that
\begin{equation}\label{e:p-G''e}
G'_e=p(G''_e)=p\big( R_u(G''_e)\rtimes G''_t\big)
     = p\big(R_u(G''_e)\big)\rtimes p(G''_t),
\end{equation}
where $p\big(R_u(G''_e)\big)\cong R_u(G''_e)$ because
by Lemma \ref{l:ker-p} \emph{}we have $\ker p\subset G''_t$.

Now consider $G_e$. Write
\[\DD=\big\{ (x I_m,\hs x^{2} I_n)\ \big|\ x\in \C^\times\big\}\subset \GG.\]

\begin{lemma}\label{l:D-G'e}
$G_e=D\hm\cdot G'_e$\hs.
\end{lemma}

\begin{proof}
Clearly $D\subseteq G_e$ for any $e\in\Hom\big(\bigwedge^2 \C^m,  \C^n\big)$,
whence $D\hm\cdot G'_e\subseteq G_e$.
Conversely, if $g=(g_m,g_n)\in G_e\subset\GL(m,\C)\times\GL(n,\C)$, we choose $x\in \C^\times$
such that $x^m=\det(g_m)$\hs, and we set
\[d=(xI_m,\hs x^{2} I_n), \quad g'=(x^{-1}g_m\hs,\hs x^{-2} g_n)=d^{-1} g.\]
Then $\det(x^{-1} g_m)=1$, and therefore we have $g'\in G'$.
Since $g'=d^{-1}g$ where $d\cdot e=e$ and  $g\cdot e=e$, we see that $g'\cdot e = e$.
Thus
$g'\in G'_e$\hs, $d\in D$, and $g=d g'$, as required.
\end{proof}

We denote $\PP_t=\DD\cdot p(\GG''_t)$.

\begin{corollary}\label{c:G''e}
$G_e=p\big(R_u(G''_e)\big)\rtimes P_t$.
\end{corollary}

\begin{proof} We use Lemma \ref{l:D-G'e}  and formula \eqref{e:p-G''e}, and obtain
\[G_e=D\hm\cdot G'_e=D\hm\cdot \big(\, p(R_u(G''_e))\rtimes p(G''_t)\,\big)=
    p\big(R_u(G''_e)\big)\rtimes \big(D\hm\cdot p(G''_t)\big)=p\big(R_u(G''_e)\big)\rtimes P_t,\]
as required.
\end{proof}

\begin{proposition}
\label{p:H1}
The inclusion $\PP_t\into \GG_e$ induces a bijection
$$\Hr^1\PP_t=\Hr^1\GG_e\hs.$$
\end{proposition}

\begin{proof}
Since the group  $p\big(R_u(\GG''_e)\big)$ is unipotent, by Sansuc's lemma we have
\[\Hr^1\PP_t=\Hr^1\big( p\big(R_u(\GG''_e)\big)\rtimes \PP_t\big).\]
By  Corollary \ref{c:G''e} we have
\[\Hr^1\big( p\big(R_u(\GG''_e)\big)\rtimes \PP_t\big)=\Hr^1\GG_e\hs,\]
as required.
\end{proof}

In order to compute $\Hr^1\PP_t$
we need the Lie algebra of $\PP_t$.
We have an isomorphism $\Lie G''_t\to \Lie (G''_t/\ker p)$ given by
\[(g_m, g_n, c_1)\mapsto (g_m, c_1 I_n+g_n).\]
Then  $\Lie P_t=\Lie D+\Lie \big( p(G''_t)\big) $,
where $\Lie D=\C\cdot (I_m,2I_n)$.
Write $\g''=\Lie(G'')$, $\g'=\Lie(G')$. For the differential
$d p :\g'' \to \g'$ we have $d p(a_m,b_n,c_1) = (a_m,b_n+c_1)$. Now
$\Lie(p(G_t'')) = dp(\Lie(G_t'')) = dp (\g_t'')$, where $\g_t''$ is the
centralizer of $t$ in $\g''$.
 We see that it is straightforward to compute using computer the Lie subalgebra
\[\Lie P_t\subset\gl(m,\C)\times\gl(n,\C).\]

For classification of $\GG(\R)$-orbits in $\YY=\Hom\big(\bigwedge^2 \R^m, \R^n\big)$,
we need $\Hr^1 \GG_e$. For this end we need a real basis of $\Lie P_t$
(which is computed by our program)
and the component group $\pi_0(P_t)$.

\begin{proposition}
\label{p:varphi-surjective}
Consider the composite homomorphism
\[ \varphi\colon G''_t\onto p(G''_t)\into P_t.\]
Then the induced homomorphism
$\varphi_*\colon\pi_0(G''_t)\to\pi_0(P_t)$
is surjective.
\end{proposition}

\begin{proof}
The homomorphism
$D\times G''_t\to D\hm\cdot p(G''_t)=P_t$ is surjective. It follows that the homomorphism of the component groups
\[ \pi_0(G''_t)\isoto\pi_0(D\times G''_t)\to \pi_0(P_t)\]
is surjective, as required.
\end{proof}

We have reduced computing $\Hr^1\GG_e$ to computing $\pi_0(P_t)$.
Moreover, in view of Proposition \ref{p:varphi-surjective},
computing $\pi_0(P_t)$ reduces to computing $\pi_0(G''_t)$\hs;
see the end of Section \ref{s:Computing-H1} below.

\section{Computing $\Hr^1\GG_e$}
\label{s:Computing-H1}

We start with a tensor $e$ from Tables 1--3. These are the tensors
in the rows 1--6 (but not 1-bis) in Table \ref{tab:reps(6,2)},
in the rows 1--12 in Table \ref{tab:reps(5,3)},
and in the rows  1--3 in Table \ref{tab:reps(4,4)}.
We can compute the stabilizer $\g_e=\Lie  G_e$ of our tensor $e$
in the real Lie algebra $\Lie \GG=\gl(m,\R)\times \gl(n,\R)$.
This is a linear problem, and we can easily do that using computer.

We need the Galois cohomology $\Hr^1\GG_e\coloneqq\Hr^1(\R, \GG_e)$.
We have a computer program \cite{dG-GalCohom} described in \cite{BG}
for calculating the Galois cohomology $\Hr^1\HH$
of a real linear algebraic group $\HH$, not necessarily connected or reductive.
We call this program ``the $\Hr^1$-program''.
We assume that $\HH\subseteq \GL(N,\C)$ (for some natural number $N$)
is a real algebraic subgroup
(that is, defined by polynomial equations with real coefficients).
The input of the $\Hr^1$-program is the real Lie algebra $\Lie \HH\subset \gl(N,\R)$
given by a (real) linear basis, and the component group $\pi_0(H)$
given by a set of representatives $h_1,\dots,h_r\in \GL(N,\C)$.

However, when working on this paper, we had only an older version of the $\Hr^1$-program.
This older $\Hr^1$-program computes $\Hr^1\HH$ when $\HH$ is {\em reductive} (not necessarily connected).
Therefore, we reduce our calculation of $\Hr^1\GG_e$ to the reductive case, see below.

Using computer, we embed our tensor $e$ into a homogeneous $\ssl_2$-triple $(h,e,f)$
 in a $\Z$-graded  Lie algebra as in Section \ref{s:theta},
and we consider the reductive group $\PP_t=\DD\cdot p(\GG''_t)\subset \GG_e$, not necessarily connected;
see Section \ref{s:e-t}.
By Proposition \ref{p:H1} we have a canonical bijection $\Hr^1\PP_t\isoto \Hr^1\GG_e$.
It remains to compute $\Hr^1\PP_t$.

In order to compute  $\Hr^1\PP_t$ using the older $\Hr^1$-program,
we need a basis of the real Lie algebra $\Lie(\PP_t)$
and the component group $\pi_0(P_t)$.
Using computer, we can easily compute a basis of $\Lie(\PP_t)$.

It remains to compute $\pi_0(P_t)$.
By Proposition \ref{p:varphi-surjective},
the homomorphism $\pi_0(G''_t)\to \pi_0(P_t)$ is surjective.
It remains to find representatives of $\pi_0(G''_t)$
up to equivalence in $P_t$. Here we say that $g_1,g_2\in G''_t$ are equivalent in $P_t$
if  $p(g_1)p(g_2)^{-1}$ is contained in the identity component $P_t^0$ of $P_t$\hs.

We have a computer program checking whether an element $h$
of a reductive algebraic group $H\subset \GL(N,\C)$
(not necessarily connected) with Lie algebra $\h$
is contained in the identity component $H^0$ of $H$.
Using this program, from a set of representatives of $\pi_0(G''_t)$
we obtain a set of representatives of $\pi_0(P_t)$.
It remains to compute $\pi_0(G''_t)$.

\section{Computing $\pi_0(G''_t)$}

The group $G''_t$ acts by conjugation on its identity component $\Gtz$,
whence we obtain a homomorphism
\[\pi_0(G''_t)= G''_t/\Gtz\,\to\,\Aut \Gtz/\Inn \Gtz= \Out \Gtz,\]
where $\Inn \Gtz$ denotes the group of inner automorphisms of $\Gtz$.
The group of outer automorphisms $\Out \Gtz$ naturally acts
on the based root datum $\BRD(\Gtz)$.
In particular, it acts on the Dynkin diagram $\cD=\Dyn \Gtz$
and on the free abelian group $X^\vee_Z=X_*(Z(\Gtz))$;
see Section \ref{s:BRD}.
We obtain a homomorphism
\begin{equation}\label{e:Aut-Aut}
 \pi_0(G''_t)\to \Aut(\Dyn \Gtz)\times \Aut(X_Z^\vee).
 \end{equation}
The action of $\pi_0(G''_t)$ on $\Dyn (G''_t)^0$ preserves
the sets of highest weights of the representations of the derived
Lie algebra $\g_t^{\prime\prime\hs\der}=[\g''_t\hs,\g''_t]$
in $U=\C^m$ and in $V=\C^n$.

The group  $\Aut(\Dyn \Gtz)$ in \eqref{e:Aut-Aut} is finite, but  $\Aut(X_Z^\vee)$
is infinite when $\dim Z(\Gtz)\ge2$.
We construct a finite subgroup of $\Aut(X_Z^\vee)$
containing the image of $\pi_0(G''_t)$ in  $\Aut(X_Z^\vee)$.
Consider the symmetric bilinear form on $\g$
\[\cF(x,y)=\Tr(xy),\quad\ x,y\in \g=\gl(m,\R)\times\gl(n,\R)\]
where $xy$ denotes the product of matrices.
This symmetric bilinear form is $G$-invariant, hence $G_t''$-invariant.
By abuse of notation, we denote again by $\cF$
the restriction of the form to $\z\coloneqq\Lie Z(\Gtz)$;
it is $G_t''$-invariant and hence $\pi_0(G_t'')$-invariant.

We observe that in all our examples, the bilinear form $\cF$
on the real vector space $\z$ is positive definite.
Indeed, one can see from Tables 1--3
that the center $\z$ of the Lie algebra $\g''_t=\Lie G_i''$ is split,
that is, it can be diagonalized over $\R$.
It follows that any matrix $x\in\z$ can be diagonalized over $\R$,
and therefore $\cF(x,x)=\Tr(x^2)\ge0$; moreover, if $x\neq 0$ then $\cF(x,x)=\Tr(x^2)>0$, as required.

We embed $X_Z^\vee \into\z$:
to any $\nu\in X_Z^\vee$ we assign the element $h_\nu=(d\nu)(1)\in\z$
as in Section \ref{s:BRD}.
We obtain a  $\pi_0(G_t'')$-invariant positive definite bilinear form $\cF_X$ on $X_Z^\vee$\hs:
\[\cF_X(\nu_1,\nu_2)=\cF(h_{\nu_1}\hs,h_{\nu_2})\quad\ \text{for}\ \, \nu_1,\nu_2\in X_Z^\vee\hs.\]
Since $\cF_X$ is definite, the group  $\Aut(X_Z^\vee\hs,\cF_X)$ is finite.
We can write the homomorphism \eqref{e:Aut-Aut} as
\begin{equation*}
\pi_0(G''_t)\,\to\, \Aut(\Dyn \Gtz)\times \Aut(X_Z^\vee\hs,\cF_X)
\end{equation*}
where both automorphism groups are finite.

\section{Details of computation of Tables 1-3}
\label{s:Details}

We describe the last two  columns in our Tables 1--3.
We take a  representative $e \in \YY=\Hom(\bigwedge^2 \UU,\VV)$
and embed it into a homogeneous $\ssl_2$-triple $t=(h,e,f)$.
We consider the centralizer (stabilizer) $\g_t''$; see Section \ref{s:Galois}.
The Lie algebra $\g_t''$ is reductive.
Let $\g_t''\hs^\der=[\g_t'',\g_t'']$ denote its derived subalgebra.
The inclusion homomorphism
\[ \g_t''\hs^\der\into \g_t''\into \Lie (\GG'')=\ssl(\UU)\times\ssl(\VV)\times \R\]
induces complex  representations of $\g_t''\hs^\der$ in $U$ and $V$.
We compute the highest weights of these representations and write them
in the corresponding columns ``Rep.~in $U$'' and ``Rep.~in $V$''
of the tables.

For example, in the row 1 of Table \ref{tab:reps(6,2)},
we have $\g_t''\simeq \ssl_2\times\ssl_2\times \ssl_2$,
and the representation of $\g_t''^\der=\g_t''$ in $U=\R^6$ is the direct sum
of the three 2-dimensional irreducible representations
with the highest weights $(1,0,0)$, $(0,1,0)$, and $(0,0,1)$.
The representation in $V$ is the trivial 1-dimensional irreducible representation
(with highest weight $(0,0,0)$\hs) with multiplicity 2.

Let $e$ be one of the tensors in the rows 2,\,3,\,4,\,5,\,6 of Table \ref{tab:reps(6,2)}.
Looking at the last three columns of the table,
we see that the Dynkin diagram $\Dyn\Gtz$
has no non-trivial automorphisms preserving the highest weights of
the representations in $U$ and $V$.
Thus $\pi_0(G''_t)$, when acting on $\Gtz$, acts trivially on $\Dyn\Gtz$.
Write $r$ for the semisimple rank of $\Gtz$, and let $x_1,\dots, x_r, y_1,\dots,y_r, h_1,\dots,h_r$
be canonical generators of the semisimple Lie algebra $[\g_t'',\g_t'']$ as in  Section \ref{s:theta}.
We add the equations
\[ \Ad(g)x_i=x_i,\ \Ad(g)y_i=y_i \quad\text{for}\ \, i=1,\dots, r\]
to the equations defining $G''_t$, and compute the Gr\"obner basis.
A calculation shows that the obtained subgroup
$H''=\ker\big[G''_t\to \Aut\g''_t\big]\subset G''$
is contained in the diagonal maximal torus of $G''$.
This means that $H''$ is given as the
intersection of a finite number of characters of the maximal torus.
This makes it possible to work with $H''$ using the machinery of
finitely-generated abelian groups. In particular we can use the Smith form of an
integral matrix to find generators of $\pi_0(H'')$. Here we do not go into the
details but refer to \cite[Proposition 3.9.7]{dG} and its proof.
A calculation shows that the images of all these representatives
are contained in the identity component  $P_t^0$ of $P_t$\hs.
By Corollary \ref{c:A1-Dyn} the image of  $\pi_0(H'')$
in $\pi_0(P_t)$ is the whole $\pi_0(P_t)$, and we conclude that
$\pi_0(P_t)=1$ in all these cases.
A calculation (by hand or using computer) shows
that $\Hr^1\PP_t=1$, and thus $\Hr^1\GG_e=1$, in all these cases.
Similarly, we obtain that for the cases
 $8,\hs 9,\hs 10,\hs 11$ of Table \ref{tab:reps(5,3)}
and the cases $2,\hs 3$ of Table \ref{tab:reps(4,4)} we have $\pi_0(P_t)=1$ and  $\Hr^1\GG_e=1$.
Thus the complex orbit $G\cdot e$ contains only one real orbit $\GG(\R)\cdot e$,
and the corresponding complex two-step nilpotent Lie algebra has only one real form.

In the case $1$  of Table \ref{tab:reps(5,3)} we have $\pi_0(P_t)=1$, but $\#\Hr^1\PP_t=2$.
Thus there are two real orbits in the complex orbit,
and we computed (using computer) a representative of the second real orbit;
see row 1-bis in Table \ref{tab:reps(5,3)}.

One can see that in the row 1 of Table \ref{tab:reps(6,2)},
the automorphism group of the Dynkin diagram $\Dyn\hs(G''_t)^0$
together with the highest weights of the representations in $U$ and $V$
is the symmetric group $S_3$.
A calculation shows that the homomorphisms
\[ \pi_0(P_t)\longleftarrow \pi_0(G''_t)\lra \Aut(\Dyn\Gtz)\]
are isomorphisms, whence $\pi_0(P_t)\simeq S_3$\hs.
A calculation (using the $\Hr^1$-program or by hand) shows
that $\Hr^1\GG_e=\Hr^1 \PP_t\cong \Hr^1(\Gamma, S_3)$ and $\#\Hr^1\GG_e=2$.
Thus there are two real orbits in the complex orbit  $G\cdot e$;
see row 1-bis in Table 1 for a representative of the second real orbit.

Similarly, in the row 1 of Table \ref{tab:reps(4,4)},
we have $\pi_0(G_2)\cong C_2$ (a group of order 2), and
\[\Hr^1\GG_e=\Hr^1\PP_t\cong \Hr^1(\Gamma,C_2)=\{[1],[c]\}\]
where $c\in C_2, c\neq 1$.
Again we have two real orbits in the complex orbit $G\cdot e$.

We consider the rows $2,3,4,5,6,7,11$ of Table \ref{tab:reps(5,3)}.
We see from the table that $\Gtz$ is a torus of dimension 2 or 3.
We computed the Gr\"obner basis of the equations defining $G''_t$.
In cases $4,6,7,11$ we obtain a subgroup
contained in the diagonal maximal torus of $G''$,
and the image of this subgroup in $P_t$ is contained in the identity component.
We see that  $\pi_0(P_t)=1$, and therefore $\PP_t$ is a split torus.
Thus $\Hr^1\GG_e=\Hr^1\PP_t=1$,
and the complex orbit $G\cdot e$ contains only one real orbit $\GG(\R)\cdot e$.

It remains to consider the cases $2,\,3,\,5$ of Table \ref{tab:reps(5,3)},
in which $\Gtz$ is a torus and the group $G_t''$ is not diagonal.
We provide details for the case 5; the cases 2 and 3 are similar (and easier).
In case 5 the group $\Gtz$ is a 3-dimensional torus, and hence $X_Z^\vee$ is
a free abelian group of rank 3.
We chose a basis $e_1,e_2,e_3$ of $X_Z^\vee$ and computed the Gram matrix
${\rm Gr}(\cF_X)=(a_{ij})$ where $a_{ij}=\cF_X(e_i,e_j)$.
We obtained the matrix
\[
\SmallMatrix{ 4 &2 &0\\ 2 &4 &0 \\ 0 &0 &31}.
\]
Using the function {\tt AutomorphismGroup} of Magma,
we computed the automorphism group ${\mathcal A}=\Aut(X_Z^\vee,\cF_X)$.
It is a group of order 24 with generators
\[
\SmallMatrix{0&1&0\\1&0&0\\ 0&0&1},\quad
\SmallMatrix{1&-1&0\\0&-1&0\\0&0&-1},\quad
\SmallMatrix{-1&0&0\\0&-1&0\\0&0&1}.
\]
We computed the list of elements of $\mathcal A$,
and for each $a\in \mathcal A$ we computed the Gr\"obner basis for $G_t''$
with additional equations saying
that the element $g\in G_t''$ acts on $X_Z^\vee$ as $a$.
For 18 elements $a$ we got the trivial Gr\"obner basis $\{1\}$,
which means that the corresponding enlarged system of equations has no solutions.
For the following 6 elements:
\[
\SmallMatrix{ -1& 0& 0 \\ -1& 1& 0 \\ 0& 0& 1 },\quad
\SmallMatrix{ -1& 1& 0\\ -1& 0& 0 \\ 0& 0& 1 },\quad
\SmallMatrix{ 0& -1& 0 \\ 1& -1& 0 \\ 0& 0& 1  },\quad
\SmallMatrix{0& 1& 0 \\ 1& 0& 0 \\ 0& 0& 1  },\quad
\SmallMatrix{  1& -1& 0 \\ 0& -1& 0 \\ 0& 0& 1},\quad
\SmallMatrix{ 1& 0& 0 \\ 0& 1& 0 \\ 0& 0& 1  }
\]
we obtained nontrivial Gr\"obner bases,
which meant that the corresponding enlarged system
did have a solution, and then it was easy to find a solution
by a computer-assisted calculation.
From this we obtained that $\pi_0(G_t'')\simeq S_3$
and $\Hr^1 \GG_e=\Hr^1\PP_t\simeq\Hr^1(\Gamma,S_3)$ is of cardinality 2.
Similarly, we obtained that $\pi_0(G_e)\simeq C_2$ in cases 2 and 3  of Table \ref{tab:reps(5,3)}.
In these two cases we also obtained that $\#\Hr^1\GG_e=2$.
Thus the complex orbit $G\cdot e$ contains exactly two real orbits,
and in each case  we computed a representative of the other orbit.

\appendix

\section{Signature $(4,4)$: the duality approach}
\label{app:Duality}

We use the duality approach; see Gauger \cite[Section 3]{Gauger}
or Galitski and Timashev \cite[Section 1.2]{GT}.
Let $U$ and $V$ be finite  dimensional spaces over a field $\kk$
of characteristic different from 2.
To each surjective linear map
$\beta\colon\bigwedge^2 U\to V$ we assign the natural surjective linear map
$$\beta^*\colon \bigwedge ^2 U^*\to  \bigwedge ^2 U^*/\hs(\ker\beta)^\bot$$
where $(\ker\beta)^\bot$ denotes the orthogonal complement (annihilator)
to $\ker\beta\subset \bigwedge^2 U$ in the dual space $\bigwedge^2 U^*$ to $\bigwedge^2 U$.
This approach reduces classification
of surjective skew-symmetric bilinear maps $\kk^m\times \kk^m\to \kk^{n_1}$
to classification  of surjective skew-symmetric bilinear maps $\kk^m\times\kk^m\to \kk^{n_2}$
where $n_2=\binom m2-n_1$.
When $(m,n_1)=(4,4)$, we obtain $n_2=\binom42-4=6-4=2$.
This reduces the problem of classification of two-step nilpotent Lie algebras over $\kk$
of signature $(4,4)$ to the well-known cases of signatures  (4,2) and (3,2)
(note that the signatures  (2,2) and (1,2) are impossible).

In \cite{dG-6} one can find a classification of 6-dimensional nilpotent Lie algebras
over a field $\kk$ of characteristic different from 2.
This gives, in particular, a classification of surjective skew-symmetric bilinear  maps
\[\beta\colon \kk^4\times\kk^4\to\kk^2.\]
over such fields. These are representatives of the orbits with the numbering of \cite[Section 4]{dG-6}.
\begin{align*}
&\beta_{6,8}=\ee125+\ee136 \quad\ \text{of signature}\ (3,2), \\
&\beta_{6,22}(\epsilon)=\ee125+\ee136+\epsilon\ee246+\ee345\ \ \text{for}\  \epsilon\in \kk \quad  \text{of signature}\ (4,2).
\end{align*}
Here $\beta_{6,22}(\delta)$ is equivalent to $\beta_{6,22}(\epsilon)$ if
$\delta=\alpha^2\epsilon$ for some $\alpha\in \kk^\times$.
Thus for $\kk=\C$ we obtain 3 equivalence classes with representatives
\[\beta_{6,8},\,\beta_{6,22}(0),\,\beta_{6,22}(1),\]
and for $k=\R$ we obtain 4 equivalence classes with representatives
\[\beta_{6,8},\,\beta_{6,22}(0),\,\beta_{6,22}(1),\, \beta_{6,22}(-1),\]
which is compatible with our Table \ref{tab:reps(4,4)}.
\bigskip

\noindent{\sc Data availability:}
 All data generated or analyzed during this study are included in this published article.

\end{document}